\documentclass[12pt]{amsart}

\usepackage{amsmath}

\def\lr{\mbox{\begin{picture}(7,10)
\put(0,8){\line(1,-1){10}}
\put(0,8){\line(1,0){10}}
\put(10,8){\line(0,-1){10}}
\end{picture}
}}

\usepackage{tikz}
\usepackage{caption}
\usepackage{subcaption}
\usepackage[colorlinks=true, pdfstartview=FitV, linkcolor=blue, citecolor=blue, plainpages=false, pdfpagelabels=true, urlcolor=blue]{hyperref}

\DeclareMathOperator{\sech}{sech}

\usepackage{subcaption}

\usepackage{graphicx}
\usepackage{color}
\usepackage{amsmath}
\usepackage{amssymb}
\newcommand{\beq}{\begin{eqnarray*}}
\newcommand{\feq}{\end{eqnarray*}}
\newcommand{\beqn}{\begin{eqnarray}}
\newcommand{\feqn}{\end{eqnarray}}

\newtheorem{theorem}{Theorem}[section]
\newtheorem{lemma}[theorem]{Lemma}

\newtheorem{proposition}[theorem]{Proposition}
\theoremstyle{definition}

\newtheorem{example}[theorem]{Example}

\theoremstyle{remark}

\numberwithin{equation}{section}

\allowdisplaybreaks

\begin{document}

\title[ Sharp Threshold for Wave Breaking in Whitham-Type Equations]{ A Sharp Threshold for Wave Breaking in Nonlocal Whitham-Type Equations}
\author{Yongki Lee}
\address{Department of Mathematical Sciences, Georgia Southern University, Statesboro,  30458}
\email{yongkilee@georgiasouthern.edu, WEB: sites.google.com/site/yongkimath/}

\keywords{Wave breaking, Whitham equation, Shallow water equations.}
\subjclass{Primary, 35L05; Secondary, 35B30}
\begin{abstract}


This paper establishes a sharp, expanded wave-breaking criterion for a class of nonlinear nonlocal Whitham-type equations, significantly generalizing the classical threshold introduced by Seliger \cite{Se68}. While the system of inequalities governing the spatial extrema of the fluid deviation has been extensively studied, establishing a sharp condition has remained a challenge. This difficulty is primarily due to the non-cooperative nature of the system, which precludes the application of standard comparison principles. We rigorously overcome this analytical obstacle by identifying the exact nonlinear threshold that confines the solutions of the inequality system and analyzing its time evolution.

\end{abstract}
\maketitle

\section{Introduction and statement of main result}
In this paper, we are concerned with the wave breaking phenomena - solutions that remain bounded while their derivatives become unbounded - for the nonlocal Whitham  type equation \cite{Wh74}:
\begin{equation}\label{whitham_0}
\left\{
  \begin{array}{ll}
   \partial_t u + u \partial_x u + \int_{\mathbb{R}} K (x-\xi)u_{\xi} (t, \xi) \, d\xi =0, & t>0, x \in \mathbb{R}, \\
    u(0,x)=u_0 (x), &   x\in \mathbb{R},\hbox{}
  \end{array}
\right.
\end{equation}
where 
$K(x)=\frac{1}{2\pi}\int_{\mathbb{R}} c(\kappa) e^{i\kappa x} \, d\kappa$
is the Fourier transform of the prescribed phase velocity $c(\kappa)$. The function $u(t,x)$ models the deviation of the fluid surface from the rest position. The equation was proposed by Whitham as an alternative to the Korteweg-de Vries (KdV) equation for describing surface wave motion in an ideal fluid.

Whitham emphasized that wave breaking is one of the most intriguing long-standing problems in water wave theory. Since the KdV equation cannot describe breaking, he proposed \eqref{whitham_0} with the singular kernel
\begin{equation}\label{singular}
K_0 (x) = \int_{\mathbb{R}} \bigg{(}\frac{\tanh \xi}{\xi}  \bigg{)}^{1/2} e^{i\xi x} \, d \xi,
\end{equation} as a model equation that combines full linear dispersion with long-wave nonlinearity, and conjectured wave breaking for \eqref{whitham_0}-\eqref{singular}.

The formal approach to proving wave breaking for Whitham-type equations originates from Seliger's 
ingenious argument \cite{Se68}, which tracks the dynamics of
\begin{equation}\label{infsup}
m_1 (t) : = \inf_{x \in \mathbb{R}} [u_x (t,x)] \ \text{and} \ m_2 (t) : = \sup_{x \in \mathbb{R}} [u_x (t,x)],
\end{equation}
attained at $x=\xi_1 (t)$ and $x=\xi_2 (t)$, respectively, provided that $K$ is bounded and integrable, among other assumptions.
The mapping $t \rightarrow \xi_i (t)$, however, may be multivalued, so these curves are not necessarily well-defined in general. Moreover, Seliger's formal analysis requires the additional assumption that the curves $\xi_1 (t)$ and $\xi_2 (t)$ are smooth. These strong assumptions were later removed in the rigorous work of Constantin and Escher \cite{CE98}.

\begin{figure}[ht]
\begin{center}
\includegraphics[width=120mm]{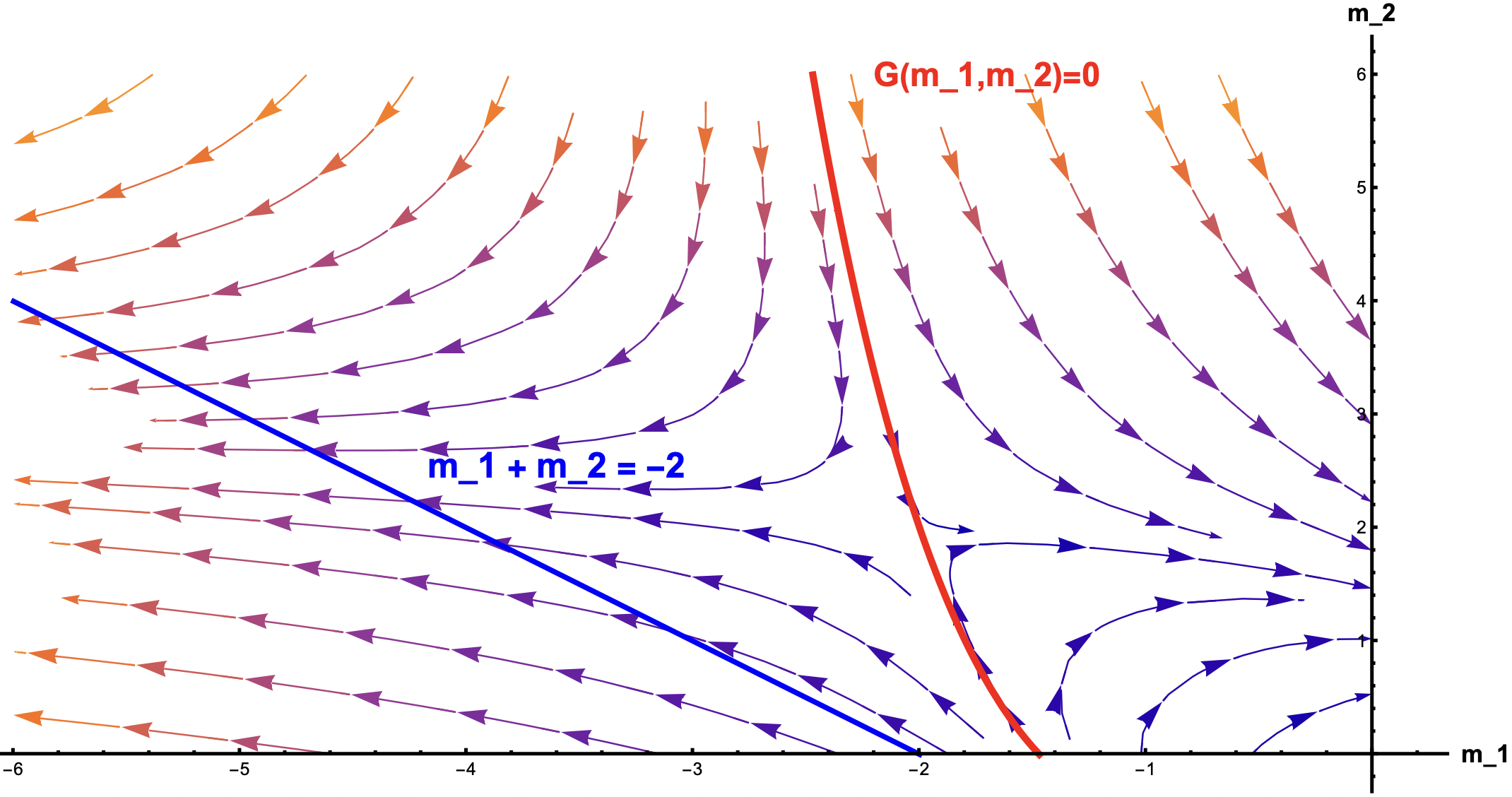}
\end{center}
\caption{The area to the left of the red curve represents $\Omega$ in Theorem \ref{main_thm}. The region below the blue line represents the wave-breaking condition \eqref{Secon} in Seliger's work \cite{Se68}. The small arrows display the phase portrait of the corresponding system of differential \emph{equations}.  }\label{fig1}
\end{figure}

Following the arguments in \cite{Se68,CE98}, we assume that $K(x)$ is smooth, integrable on $\mathbb{R}$, symmetric, and monotonically decreasing on $x\in [0,\infty)$. For non-integrable kernels, we refer to \cite{VH17} and references therein.

Differentiating the first equation in \eqref{whitham_0} with respect to $x$ and evaluating the resulting expressions at $x=\xi_1 (t)$ and $x=\xi_2 (t)$, one obtains the following coupled differential inequalities, as derived in \cite{Se68}:
\begin{equation}\label{m_eqn_k}
\left\{
  \begin{array}{ll}
  m'_1 (t)\leq -m^2 _1 (t) + K(0)(m_2 (t) -m_1 (t)) \ \ a.e, \\
   m'_2 (t) \leq -m^2 _2 (t) + K(0)(m_2 (t)-m_1 (t)) \ \ a.e.
  \end{array}
\right.
\end{equation}
Here, $K(0)>0$.

The wave breaking condition for the Whitham-type equation in \cite{Se68} is 
\begin{equation}\label{Se68con}
\inf_{x \in \mathbb{R}} [u' _0 (x)]+\sup_{x \in \mathbb{R}} [u' _0 (x)]\leq -2K(0),
\end{equation}
which states that a sufficiently asymmetric initial profile leads to wave breaking in finite time. 

The above argument, together with the rigorous proof in \cite{CE98}, has become a cornerstone in the study of wave breaking for nonlinear nonlocal shallow water equations. Moreover, condition \eqref{Se68con}, or similar methods based on bounding linear combinations of extrema, has been widely used in the literature, including recent works \cite{SH17, GH18,  YL20, MLQ16, LW18}. This approach is effective because linear combinations of spatial extrema preserve a useful structure: for instance, if $m_1 (0) +m_2 (0) \leq -2K(0)$, then this inequality remains valid for all time.

The main contribution of this study is a rigorous derivation of a significantly more general and sharp wave breaking criterion that strictly extends Seliger's classical condition \eqref{Se68con}. We emphasize that the system \eqref{m_eqn_k} is non-cooperative  \cite{EK39, HS95,WW70}; therefore, standard comparison arguments with the corresponding equality system are inapplicable.




To state the main result, for simplicity we normalize $K(0)= 1$, so that \eqref{m_eqn_k} reduces to
\begin{equation}\label{m_eqn}
\left\{
  \begin{array}{ll}
  m'_1 (t)\leq -m^2 _1 (t) + m_2 (t) -m_1 (t) \ \ a.e, \\
   m'_2 (t) \leq -m^2 _2 (t) +m_2 (t)-m_1 (t) \ \ a.e.
  \end{array}
\right.
\end{equation}

From \eqref{infsup}, we necessarily have $m_1 (t) \leq 0 \leq m_2 (t)$, as long as they exist. 

We now define
$$
 \Omega:=\{(x,y)\in \mathbb{R}^- \times \mathbb{R}^+  \, \big| \, G(x,y)<0 \},
$$
where
$$
G(x,y):=x+y+\text{sgn}(x-y+4)\sqrt{(x-y)\big\{(x-y)+4\ln(y-x)+4-4\ln 4 \big\}},
$$
and state the main theorem.

\begin{theorem}\label{main_thm}
Consider \eqref{whitham_0}. If $u_0 \in H^2(\mathbb{R})$ satisfies
$$(\inf_{x \in \mathbb{R}}[u' _0(x)] , \sup_{x \in \mathbb{R}}[u' _0 (x)]) \in \Omega,$$
then
$$
\lim_{t\to T-} \inf_{x\in\mathbb{R}}[u_x(t,x)] =-\infty,
$$
where 
$$T\leq -\frac{2}{G(\inf_{x \in \mathbb{R}}[u' _0(x)], \sup_{x \in \mathbb{R}}[u' _0(x)])}.$$
\end{theorem}

$$$$

Several remarks are in order:

(i) The wave breaking condition in Seliger \cite{Se68} is
\begin{equation}\label{Secon}
m_1(0)+m_2(0)\leq -2.
\end{equation}
As illustrated in Figure \ref{fig1}, Seliger's criterion captures the region below the solid blue line. Our extended condition $G(m_1 , m_2)<0$ strictly encompasses Seliger's condition and the boundary curve $G(m_1, m_2)=0$ precisely traces the exact stable manifold connected to the hyperbolic saddle point $(-2,2)$ of the associated auxiliary equality system.

(ii) A natural approach to analyzing the inequality system \eqref{m_eqn} would be to employ standard comparison principles, such as Kamke's comparison theorem \cite{EK39, HS95,WW70} to directly bound its trajectories by those of the corresponding equality system. However, one cannot trivially assert that the trajectories of the inequality system are pointwise bounded by those of the equality system. More precisely, Kamke's theorem explicitly requires the underlying vector field to be cooperative (or quasi-monotone non-decreasing); that is, all off-diagonal elements of the system's Jacobian matrix must be non-negative. For the vector field defined in \eqref{m_eqn}, the off-diagonal derivative is strictly negative: $\frac{\partial}{\partial m_1}(-m^2 _2 +m_2 -m_1)=-1$. Thus, the system is strictly non-cooperative, rendering standard comparison principles inapplicable.

(iii) This inherent structural lack of monotonicity underscores the necessity of our geometric approach. To rigorously overcome this obstacle and force finite-time blow-up, we explicitly construct the sharp one-way trapdoor $G(x,y)=0$ for the inequality system (Lemma \ref{lemma1}) alongside the study of the time evolution of $G(x(t), y(t))$ (Lemma \ref{lemma2}).\\

The details of the proof of Theorem \ref{main_thm} are carried out in the rest of the paper.

\section{Proof of Theorem \ref{main_thm}}
We first consider the following auxiliary equality system on the domain $\mathbb{R}^- \times \mathbb{R}^+ :=(-\infty, 0)\times(0,\infty)$.
\begin{equation}\label{auxsystem}
\left\{
  \begin{array}{ll}
  x'(t)=-x^2(t)+y(t)-x(t), \\
   y'(t)=-y^2(t)+y(t)-x(t).
  \end{array}
\right.
\end{equation}
This system possesses two equilibrium points,
$$
(x,y)=(-2,2) \text{ and } (0,0),
$$
where the first is a hyperbolic saddle and the second is a degenerate equilibrium of saddle type. We establish a sharp blow-up condition for this auxiliary system.

\begin{figure}[h!]
\centering
\begin{subfigure}{0.48\textwidth}
    \centering
    \includegraphics[width=\textwidth]{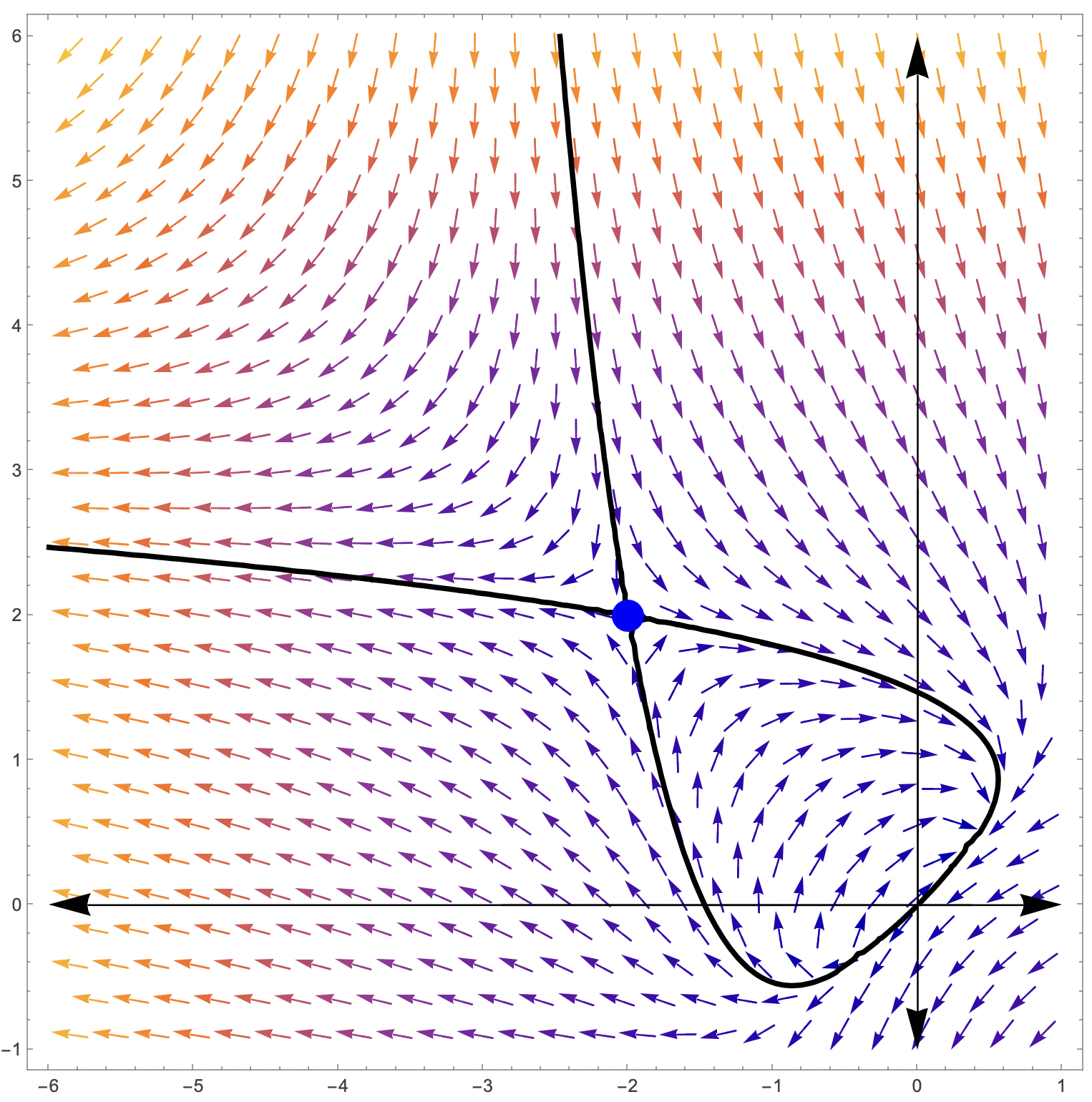}
    \caption{The phase plane of system \eqref{auxsystem} and the graph of the first integral  \eqref{1stint}-\eqref{Kval}}
    \label{fig:figure1}
\end{subfigure}
\hfill
\begin{subfigure}{0.48\textwidth}
    \centering
    \includegraphics[width=\textwidth]{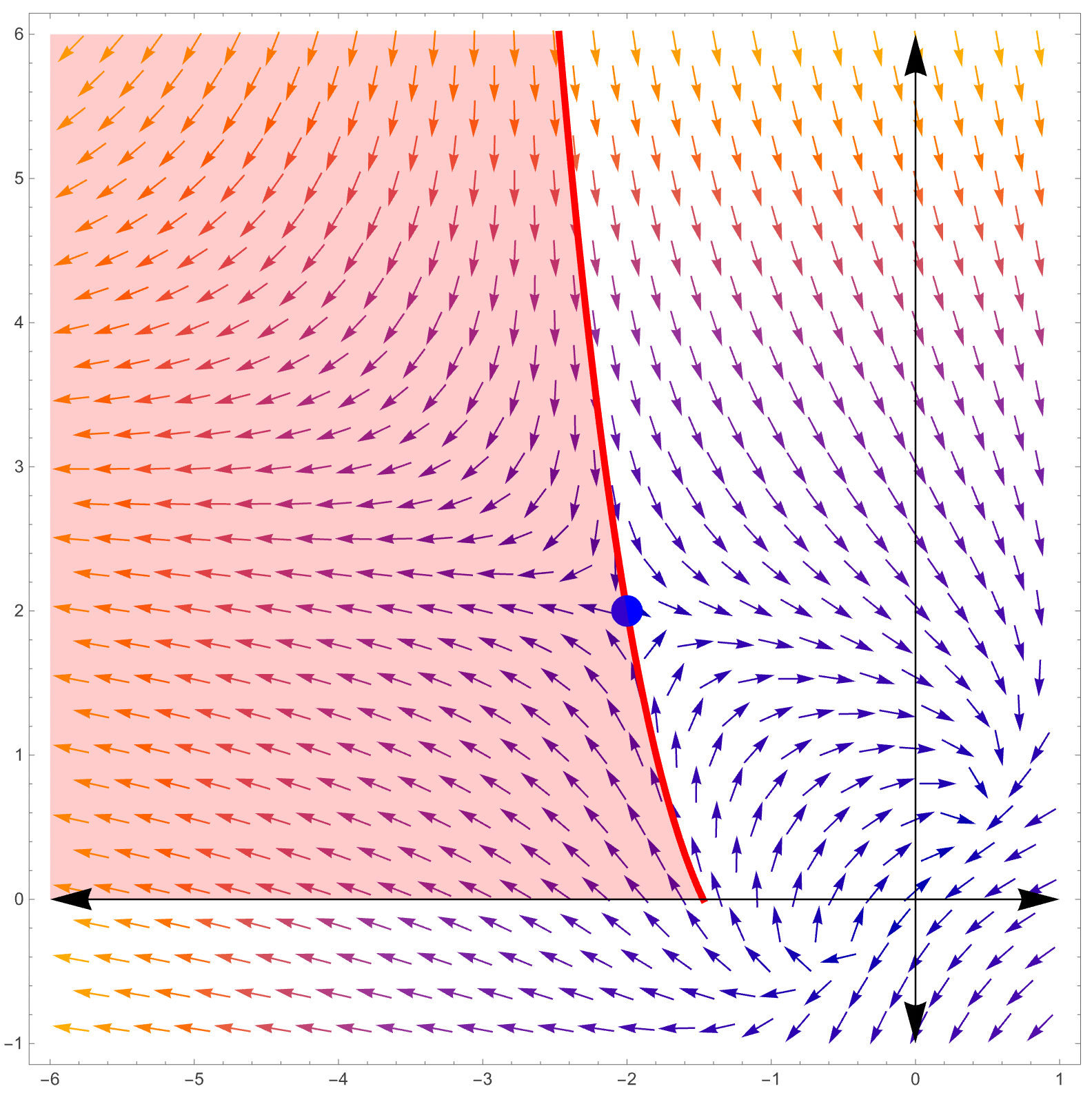}
    \caption{The separatrix \eqref{sep} and the blow-up region for the system \eqref{auxsystem}.}
    \label{fig:figure2}
\end{subfigure}
\caption{Phase plane of the auxiliary system \eqref{auxsystem}.}
\end{figure}

\begin{proposition}
Consider system \eqref{auxsystem} with initial conditions $(x(0), y(0)) \in \mathbb{R}^- \times \mathbb{R}^+$. It holds that $x(t)\to -\infty$ in finite time if and only if
$$
(x(0) , y(0))\in \Omega:=\{(x,y)\in \mathbb{R}^- \times \mathbb{R}^+  \, \big| \, G(x,y)<0 \},
$$
where
\begin{equation}\label{Gsep}
G(x,y):=x+y+\text{sgn}(x-y+4)\sqrt{(x-y)\big\{(x-y)+4\ln(y-x)+4-4\ln 4 \big\}}.
\end{equation}
\end{proposition}

\begin{proof} We first identify the separatrix of the system. Introduce the substitutions
$$
u(t):=x(t)-y(t) \text{ and } v(t):=x(t)+y(t).
$$
Computing the derivatives of these yield
$$
\begin{aligned}
u'&=x'-y'\\
&=y^2 -x^2\\
&=-uv,
\end{aligned}
$$
and
$$
\begin{aligned}
v'&=x'+y'\\
&=-(x^2 + y^2)+2(y-x)\\
&=-\frac{u^2+v^2}{2}-2u.
\end{aligned}
$$
Using these derivative, we obtain
$$
\frac{dv}{du}=\frac{v'}{u'}=\frac{u^2+v^2+4u}{2uv}.
$$
Equivalently, this can be written as
$$
2uv \, dv - v^2 \, du = u^2 \, du + 4u \, du
$$
or
$$
\frac{2uv \, dv - v^2 \, du}{u^2}=\bigg(1+\frac{4}{u} \bigg) \, du.
$$
Integrating both sides yields
$$
\frac{v^2}{u}=u+4(\ln |u|)+C,
$$
where $C$ is a constant of integration.

Substituting the original variables back into the equation gives
\begin{equation}\label{0thint}
\frac{(x+y)^2}{x-y}=x-y + 4\ln |x-y|+C.
\end{equation}
After further simplification, we obtain the following first integral of system \eqref{auxsystem}.
\begin{equation}\label{1stint}
xy=(x-y)(\ln(y-x) + K),
\end{equation}
where $K=\frac{C}{4}$ is a constant. Because we are restricting our domain to $(x,y)\in \mathbb{R}^- \times \mathbb{R}^+$, we have $x<y$, allowing $|x-y|$ to simplify to $y-x$.

Substituting the coordinates of the saddle point $(x,y)=(-2,2)$ into \eqref{1stint} and solving for $K$ gives
\begin{equation}\label{Kval}
K=1-\ln(4).
\end{equation}

The phase portrait of system \eqref{auxsystem}, along with the contour of the first integral defined by \eqref{1stint} and \eqref{Kval}, are shown in Figure \ref{fig:figure1}.

The stable manifold curve associated with the saddle point $(-2,2)$ is obtained by rearranging \eqref{0thint}, substituting $C=4K=4-4\ln 4$, and selecting the upper and lower branches of the first integral:
\begin{equation}\label{sep}
x+y=-\text{sgn}(x-y+4)\sqrt{(x-y)\big\{(x-y)+4\ln(y-x)+4-4\ln 4 \big\}}.
\end{equation}
Define
$$
G(x,y):=x+y+\text{sgn}(x-y+4)\sqrt{(x-y)\big\{(x-y)+4\ln(y-x)+4-4\ln 4 \big\}}.
$$

This sharp separatrix, defined by $G(x,y)=0$, is shown in Figure \ref{fig:figure2}. The $x$-intercept of the separatrix is $(-4/e, 0)$. 
The remainder of the proof relies on standard phase-plane arguments: the positive invariance of the separatrix $G(x,y)=0$ strictly confines trajectories starting within $\Omega$, where the dominance of the $-x^2$ term in the vector field guarantees finite-time blow-up.
\end{proof}

Now, we consider the original system of differential inequalities
\begin{equation}\label{ineqsystem}
\left\{
  \begin{array}{ll}
  x'(t)\leq F(x(t),y(t)):= -x^2(t)+y(t)-x(t), \\
   y'(t)\leq H(x(t), y(t)):=-y^2(t)+y(t)-x(t).
  \end{array}
\right.
\end{equation}

We point out that all physically relevant trajectories naturally lie in the second quadrant. This is because
$$
\inf_{x\in \mathbb{R}}[u_x (t,x)]=x(t)\leq 0\leq y(t)=\sup_{x\in \mathbb{R}}[u_x (t,x)].
$$
Therefore, we restrict our analysis to this region.

Before proceeding with the formal phase plane analysis, we briefly remark on the ``almost everywhere" ($a.e.$) condition present in the differential inequalities \eqref{m_eqn}. For a sufficiently regular solution $u(t,x)$, it is a standard result (see, e.g., \cite{CE98}) that the spatial extremum functions $m_1(t)$ and $m_2(t)$ are locally Lipschitz continuous in time. Consequently, they are absolutely continuous, and by Rademacher's theorem, their classical time derivatives exist almost everywhere. The absolute continuity of these functions guarantees that the Fundamental Theorem of Calculus holds, allowing us to rigorously integrate these differential inequalities over any time interval. Thus, the $a.e.$ restriction presents no analytical obstacle.

We first prove that the separatrix $G(x,y)=0$ of the equality system \eqref{auxsystem} serves as a one-way trapdoor.

\begin{lemma}\label{lemma1}
Consider the system \eqref{ineqsystem}. Let $(x(t), y(t))$ be a trajectory of the inequality system. If the initial condition satisfies $(x(0), y(0))\in \Omega$, then 
$$
(x(t), y(t))\in \Omega
$$
for all $t\geq 0$, as long as the solution exists. In other words, the trajectory $(x(t), y(t))$ cannot cross the curve $G(x,y)=0$, $x<-\frac{4}{e}$.
\end{lemma} 

\begin{proof}
We rewrite the inequality system as a system of exact differential equations by introducing nonnegative, time-dependent functions $a(t)\geq 0$ and $b(t)\geq 0$:
\begin{equation}\label{tempsystem}
\left\{
  \begin{array}{ll}
  x'(t)=F(x(t),y(t))-a(t), \\
  y'(t)=H(x(t), y(t))-b(t).
  \end{array}
\right.
\end{equation}

Let $y = g(x)$ denote the separatrix curve $G(x,y)=0$ of the equality system \eqref{auxsystem} obtained in the previous Lemma. That is, $y = g(x)$ is obtained by solving $G(x, y)=0$ for $y$.

To prove that a trajectory cannot escape through the curve $y = g(x)$, we define the vertical distance function
$$
V(t)=y(t)-g(x (t)).
$$

In the region $\Omega$, we have $V(t)<0$. If a trajectory touches the boundary $y= g(x)$, then $V(t)=0$.

Consider
\begin{equation}\label{hprime}
\begin{aligned}
V'(t)&=y'(t)-g'(x(t))x' (t)\\
&=\big\{ H(x(t), y(t))-b(t) \big\} - g'(x(t)) \big\{F(x(t),y(t))-a(t) \big\}\\
&=\big\{H(x(t), y(t)) -g'(x(t))F(x(t), y(t)) \big\}-b(t)+a(t)g'(x(t)).
\end{aligned}
\end{equation}

If $V(t)=0$, or equivalently $y(t) = g(x(t))$, then 
$$
g'(x(t))=\frac{dy}{dx}=\frac{H(x(t), g(x(t)))}{F(x(t), g(x(t)))}.
$$
Substituting this expression into \eqref{hprime} yields
$$
V'(t)=-b(t)+a(t)g'(x(t)).
$$
Since $g'(x(t))<0$, we conclude that
$$
V'(t)\leq 0
$$
along the curve $y = g(x)$. Therefore, whenever $V(t)=0$, we have $V'(t)\leq 0$. Hence, the trajectory of the inequality system cannot cross the curve $y = g(x)$.
\end{proof}

Next, we study the time evolution of $G(x(t), y(t))$, where $G(x,y)$ is defined in \eqref{Gsep}.

\begin{lemma}\label{lemma2}
Consider the system \eqref{ineqsystem}. Let $(x(t), y(t))$ be a trajectory of the inequality system. If the initial condition satisfies $(x(0), y(0))\in \Omega$, then
$$
\frac{d}{dt}G(x(t), y(t)) \leq -\frac{1}{2}G^2(x(t), y(t)),
$$
for all $t\geq 0$, as long as the solution exists.
\end{lemma}

\begin{proof} Note that the region $\Omega$ is defined by $G(x,y)<0$ and $x<-\frac{4}{e}$.
We define 
$$
z(t):=y(t)-x(t), \text{ and } v(t):=x(t)+y(t).
$$
Then, $G(x(t),y(t))$ can be written as
$$
G(x(t),y(t))=v(t)+\text{sgn}(4-z(t))\sqrt{z^2(t) -4z(t)\ln(z(t)/4)-4z(t)},
$$
or equivalently,
\begin{equation}\label{GvS}
G(x(t), y(t))=v(t)+S(z(t)),
\end{equation}
where
\begin{equation}\label{SWdef}
S(z):=\text{sgn}(4-z)\sqrt{W(z)}, \text{ and } W(z):=z^2 -4z\ln(z/4)-4z.
\end{equation}

We first claim that
\begin{equation}\label{Gineq1}
\begin{aligned}
\frac{d}{dt}G(x(t), y(t))&=G_x (x(t), y(t))x'(t)+G_y(x(t), y(t))y'(t)\\
&\leq (1-S'(z(t)))F(x(t), y(t))+(1+S'(z(t)))H(x(t), y(t)).
\end{aligned}
\end{equation}
To prove this, we show that
$$
G_x(x,y)=1-S'(z) \text{ and } G_y (x,y)=1+S'(z)
$$
are nonnegative for $(x,y) \in \Omega$. This requires $(S'(z))^2\leq 1$, or equivalently 
$$
(W'(z))^2 \leq 4W(z),
$$
since $S^2(z)=W(z)$ implies $S'(z)=\frac{W'(z)}{2S(z)}$.
Substituting the expression for $W(z)$ and $W'(z)=2z-4\ln(z/4)-8$ into the above inequality yields
$$
z-\big(\ln(\frac{z}{4})+2 \big)^2\geq 0.
$$
One can easily verify that for $z>\frac{4}{e}$, the function on the left-hand side is strictly convex and attains its global minimum value $0$ at $z=4$. Hence, the inequality in \eqref{Gineq1} holds throughout $\Omega$.

Since
$$F(x,y)+H(x,y)=-\frac{1}{2}v^2 -\frac{1}{2}z^2 +2z \text{ and } H(x,y)-F(x,y)=-zv,$$ 
inequality \eqref{Gineq1} can be rewritten as
$$
\frac{d}{dt}G(x(t), y(t))\leq -\frac{1}{2}v^2(t) -\frac{1}{2}z^2(t) +2z(t)-S'(z(t))z(t)v(t).
$$
Substituting $v(t)=G(x(t), y(t))-S(z(t))$ from \eqref{GvS}, and suppressing the dependence on $x(t)$, $y(t)$, and $t$ for simplicity, we obtain
\begin{equation}\label{Gineq2}
\begin{aligned}
\frac{d}{dt}G &\leq \bigg(-\frac{1}{2}S^2 -\frac{1}{2}z^2 + 2z + zSS' \bigg)-\frac{1}{2}G^2 + G\cdot(S-zS')\\
&=-\frac{1}{2}G^2 + G(S-zS'),
\end{aligned}
\end{equation}
where the equality follows from
$$
\begin{aligned}
-\frac{1}{2}S^2 -\frac{1}{2}z^2 + 2z + zSS'&=-\frac{1}{2}W-\frac{1}{2}z^2+2z+zS(\frac{W'}{2S})\\
&=-\frac{1}{2}(z^2 -4z\ln(z/4)-4z)-\frac{1}{2}z^2+2z+z\cdot \frac{2z-4\ln(z/4)-8}{2}\\
&=0.
\end{aligned}
$$

Moreover,
$$
\begin{aligned}
G(S-zS')&=G\cdot\bigg(S-z\cdot \frac{W'}{2S}\bigg)\\
&=G\cdot\bigg(\frac{2W-zW'}{2S} \bigg)\\
&=G\cdot \bigg( \frac{-2z\ln(z/4)}{S} \bigg).
\end{aligned}
$$
Therefore, \eqref{Gineq2} reduces to
$$
\begin{aligned}
\frac{d}{dt}G(x(t), y(t))&\leq -\frac{1}{2}G^2(x(t),y(t))+G(x(t),y(t)) \bigg( \frac{-2z(t)\ln(z(t)/4)}{S(z(t))} \bigg).
\end{aligned}
$$
Finally, since 
$$G(x,y) \bigg( \frac{-2z\ln(z/4)}{S(z)} \bigg)= G(x,y) \bigg( \frac{-2(y-x)\ln((y-x)/4)}{S(y-x)} \bigg) \leq 0,$$ 
on $(x,y)\in\Omega$, we obtain the desired result.
\end{proof}

Next, we prove that $G(x(t), y(t))\to -\infty$ implies $x(t)\to -\infty$. 

\begin{lemma}\label{lemma3}
Consider the system \eqref{ineqsystem}. Suppose that $(x(0), y(0))\in \Omega$.
If there exists $T\in(0,\infty)$ such that
$$
\lim_{t\to T-}G(x(t), y(t)) =-\infty,
$$
then
$$
\lim_{t\to T-}x(t)=-\infty.
$$
\end{lemma}

\begin{proof}
Using \eqref{GvS} and $z=y-x$, we can rewrite $G(x,y)$ as
$$
G(x,y)=2x+z+S(z).
$$
Rearranging gives
$$
x=\frac{1}{2}G(x,y)-\frac{1}{2}\big(z+S(z) \big).
$$
We now verify that $z+S(z) \geq 0$ on $\Omega$.\\

\textbf{Case 1:} $z\in [4/e , 4]$.\\
From \eqref{SWdef}, $S(z)=\text{sgn}(4-z)\sqrt{W(z)}$. Since $z\in [4/e , 4]$, we have
$$
z+S(z)\geq z \geq \frac{4}{e}>0.
$$

\textbf{Case 2:} $z \in (4,\infty)$.\\
In this case, $S(z)=-\sqrt{W(z)}$. By \eqref{SWdef}, 
$$
W(z)=z^2-4z\big(\ln(z/4)+1 \big).
$$
Since $z>4$, we have $W(z)<z^2$. Therefore,
$$
z+S(z)=z-\sqrt{W(z)}>0.
$$

Hence, we established that 
$$
x(t)< \frac{1}{2}G(x(t), y(t)),
$$
which yields the desired result.
\end{proof}

We are now ready to prove Theorem \ref{main_thm}. If $(x(0),y(0))\in \Omega$, then by Lemma \ref{lemma1}, the trajectory $(x(t),y(t))$ never escapes $\Omega$. Next, integrating the inequality in Lemma \ref{lemma2} gives
$$
\frac{1}{G(x(t), y(t))}\geq \frac{1}{G(x(0), y(0))}+\frac{1}{2}t.
$$
Since $G(x(0), y(0))<0$, the above inequality implies that 
$$G(x(t), y(t)) \to -\infty \text{ as } t\to -\frac{2}{G(x(0),y(0))}.$$
Finally, Lemma \ref{lemma3} implies that $x(t)\to -\infty$. This completes the proof of Theorem \ref{main_thm}.

\bibliographystyle{abbrv}

\end{document}